\documentclass[reqno]{amsart}

\usepackage[usenames,dvipsnames]{pstricks}
\usepackage{epsfig}

\usepackage{color}
\date{\today}

\usepackage[latin1]{inputenc}
\usepackage[T1]{fontenc}
\usepackage{amsmath, amsthm}
\usepackage[english]{babel}
\usepackage{amssymb}
\usepackage{latexsym}
\usepackage{graphicx}
\usepackage[all]{xy}

\newtheorem{theorem}{Theorem}[section]

\newtheorem{proposition}[theorem]{Proposition}
\theoremstyle{definition}
\newtheorem{definition}[theorem]{Definition}

\newcommand{\ot}{\otimes}
\newcommand{\co}{\circ}

\begin{document}

\begin{center}

{\huge{\bf Strong Hopf modules for weak Hopf quasigroups}}

\end{center}

\ \\
\begin{center}
{\bf J.N. Alonso \'Alvarez$^{1}$, J.M. Fern\'andez Vilaboa$^{2}$, R.
Gonz\'{a}lez Rodr\'{\i}guez$^{3}$}
\end{center}

\ \\
\hspace{-0,5cm}$^{1}$ Departamento de Matem\'{a}ticas, Universidad
de Vigo, Campus Universitario Lagoas-Marcosende, E-36280 Vigo, Spain
(e-mail: jnalonso@ uvigo.es)
\ \\
\hspace{-0,5cm}$^{2}$ Departamento de \'Alxebra, Universidad de
Santiago de Compostela.  E-15771 Santiago de Compostela, Spain
(e-mail: josemanuel.fernandez@usc.es)
\ \\
\hspace{-0,5cm}$^{3}$ Departamento de Matem\'{a}tica Aplicada II,
Universidad de Vigo, Campus Universitario Lagoas-Mar\-co\-sen\-de, E-36310
Vigo, Spain (e-mail: rgon@dma.uvigo.es)
\ \\

{\bf Abstract} In this paper  we introduce the category of strong Hopf modules for a weak Hopf quasigroup $H$ in a braided monoidal category. We also prove that this category is equivalent to the category of right modules over the image of the target morphism of $H$.

\vspace{0.5cm}



\section{introduction}
Let ${\Bbb F}$ be a field and  ${\mathcal C}={\Bbb F}-Vect$. Let $M$ be a right $H$-module and a right $H$-comodule. If, for all $m\in M$ and $h\in H$, we write $m.h$ for the action and  we use the Sweedler notation $\rho_{M}(m)=m_{[0]}\ot m_{[1]}$ for the coaction, we will say that $M$ is a Hopf module if the equality  
$$\rho_{M}(m.h)=m_{[0]}.h_{(1)}\ot m_{[1]}h_{(2)}$$
holds, where $\delta_{H}(h)=h_{(1)}\ot h_{(2)}$ is the coproduct of $H$ and $m_{[1]}h_{(2)}$ the product in $H$ of $m_{[1]}$ and $h_{(2)}$. A morphism between two Hopf modules is a ${\Bbb F}$-linear map that is $H$-linear and $H$-colinear. Hopf modules and morphisms of Hopf modules constitute the category of Hopf modules denoted by ${\mathcal M}^{H}_{H}$. In 1969 Larsson and Sweedler proved a result, called Fundamental Theorem of Hopf Modules, that asserts the following: If $M\in {\mathcal M}^{H}_{H}$ and $M^{co H}=\{m\in M \;| \; \rho_{M}(m)=m\ot 1_{H}\}$ are the coinvariants of $H$ in $M$, $M$ is isomorphic to  $M^{co H}\ot H$ as Hopf modules (see \cite{Larson-Sweedler} and \cite{Sweedler}). On the other hand, if $N$ is a ${\Bbb F}$-vector space, the tensor product $N\ot H$, with the action and coaction induced by the product and the coproduct of $H$, is a Hopf module. This construction is functorial, so we have a functor $F=-\ot H:{\mathcal C}\rightarrow {\mathcal M}^{H}_{H}$. Also, for all $M\in {\mathcal M}^{H}_{H}$, the construction of $M^{co H}$ is functorial and we have a functor $G=(\;\;)^{co H}: {\mathcal M}^{H}_{H}\rightarrow {\mathcal C}$ and $F\dashv G$. Moreover, $F$ and $G$ is a pair of inverse equivalences, and therefore, ${\mathcal M}^{H}_{H}$ is equivalent to the category of ${\Bbb F}$-vector spaces.

The Fundamental Theorem of Hopf Modules also holds for weak Hopf algebras as was proved by B\"{o}hm, Nill and Szlach\'anyi in \cite{bohm}. In this case, if $H$ is a weak Hopf algebra, the category of Hopf modules is defined in the same way as in the Hopf algebra setting. For $M\in {\mathcal M}^{H}_{H}$, the coinvariants of $H$ in $M$ are defined by $M^{co H}=\{m\in M \;| \; \rho_{M}(m)=m_{[0]}\ot \Pi_{H}^{L}(m_{[1]})\}$, where $\Pi_{H}^{L}$ is the target morphism associated to $H$. Then, B\"{o}hm, Nill and Szlach\'anyi proved that 
$M$ is isomorphic to  $M^{co H}\ot_{H_{L}} H$ as Hopf modules, where $H_{L}$ is the image of 
$\Pi_{H}^{L}$. Moreover, if ${\mathcal C}_{H_{L}}$ is the category of right $H_{L}$-modules, there exist two functors $F=-\ot_{H_{L}} H:{\mathcal C}_{H_{L}}\rightarrow {\mathcal M}^{H}_{H}$ and $G=(\;\;)^{co H}: {\mathcal M}^{H}_{H}\rightarrow {\mathcal C}_{H_{L}}$ such that $F$ is left adjoint of $G$ and they induce a pair of inverse equivalences (see \cite{Hanna}). Therefore, in the weak setting, ${\mathcal M}^{H}_{H}$ is equivalent to ${\mathcal C}_{H_{L}}$. In this case, is a relevant fact the following property: the tensor product $M^{co H}\ot_{H_{L}} H$ is isomorphic as Hopf modules to $M^{co H}\times H$ where $M^{co H}\times H$ is the image of a suitable idempotent $\nabla_{M}:M^{co H}\otimes H\rightarrow M^{co H}\otimes H$.  Note that, as a consequence, in the weak framework the Fundamental Theorem of Hopf Modules can be written using  $M^{co H}\times H$ instead of $M^{co H}\ot_{H_{L}} H$.

In the two previous paragraphs we  spoke about associative algebraic structures like Hopf algebras and weak Hopf algebras. Recently Klim and Majid introduced in \cite{Majidesfera} the notion of Hopf quasigroup as a generalization of Hopf algebras in the context of non-associative algebra, in order to understand the structure and relevant properties of the algebraic $7$-sphere. A Hopf quasigroup is a particular instance of  unital coassociative $H$-bialgebra in the sense of P\'erez Izquierdo \cite{PI2}, and it includes as example the enveloping algebra  of a Malcev algebra, when the base ring has characteristic not equal to $2$ nor $3$. In this sense Hopf quasigroups extend the notion of Hopf algebra in a parallel way that Malcev algebras extend the one of Lie algebra. On the other hand, it also contains as an example the notion of quasigroup algebra  of an I.P. loop. Therefore, Hopf quasigroups unify I.P. loops and Malcev algebras in the same way that Hopf algebras unify groups and Lie algebras. For these non-associative algebraic structures, Brzezi\'nski introduced in \cite{Brz} the notion of Hopf module and he proved a version of  the Fundamental Theorem of Hopf Modules. In this case, the main difference appears in the definition of the category of Hopf modules ${\mathcal M}^{H}_{H}$, because the notion of Hopf module reflects the non-associativity of the product defined on $H$, and the morphisms are $H$-quasilinear and $H$-colinear  (see Definition 3.4 of \cite{Brz}). In Lemma 3.5 of \cite{Brz}, we can find that, if $M\in {\mathcal M}^{H}_{H}$ and $M^{co H}$ is defined like in the Hopf algebra setting, $M$ is isomorphic to  $M^{co H}\ot H$ as Hopf modules. Moreover, there exist two functors $F=-\ot H:{\mathcal C}\rightarrow {\mathcal M}^{H}_{H}$ and $G=(\;\;)^{co H}: {\mathcal M}^{H}_{H}\rightarrow {\mathcal C}$ such that  $F\dashv G$, and they induce a pair of inverse equivalences. Therefore, in this non-associative context ${\mathcal M}^{H}_{H}$ is equivalent to the category of ${\Bbb F}$-vector spaces as in the Hopf algebra ambit.

Working in a monoidal setting, in \cite{Asian} we introduce the notion of weak Hopf quasigroup as a new Hopf algebra generalization that  encompass weak Hopf algebras and Hopf quasigroups. A family of non-trivial examples of these algebraic structures can be obtained  working with bigroupoids, i.e., bicategories where every $1$-cell is an equivalence and every $2$-cell is an isomorphism (see Example 2.3 of \cite{Asian}).  For a weak Hopf quasigroup $H$ in a  braided monoidal category ${\mathcal C}$ with tensor product $\ot$, using the ideas proposed by Brzezi\'nski for Hopf quasigroups, in \cite{Asian} we introduce  the notion of Hopf module and  the category of Hopf modules ${\mathcal M}^{H}_{H}$. In this case, if we define $M^{co H}$ in the same way as in the weak Hopf algebra setting, we obtain a version of the  Fundamental Theorem of Hopf Modules in the following way: all Hopf module $M$ is isomorphic to  $M^{co H}\times H$ as Hopf modules, where $M^{co H}\times H$ is the image of the  same idempotent $\nabla_{M}$ used for Hopf modules associated to a weak Hopf algebra.  Moreover, in \cite{JPAA} we proved that $H_{L}$, the image of the target morphism, is a monoid and then it is possible to take into consideration the category ${\mathcal C}_{H_{L}}$, to construct the tensor product $M^{co H}\ot_{H_{L}} H$, and, if the functor $-\ot H$ preserves coequalizers, to endow this object with a Hopf module structure. Unfortunately, it is not possible to assure that $M^{co H}\ot_{H_{L}} H$ is isomorphic to $M^{co H}\times H$ as in the weak Hopf algebra case. In this paper we find the conditions under which these objects are isomorphic in ${\mathcal M}^{H}_{H}$. Then, as a consequence, we introduce the category of strong Hopf modules, denoted by   ${\mathcal SM}^{H}_{H}$ and we obtain that there exist two functors $F=-\ot_{H_{L}} H:{\mathcal C}_{H_{L}}\rightarrow {\mathcal SM}^{H}_{H}$ and  $G=(\;\;)^{co H}: {\mathcal SM}^{H}_{H}\rightarrow {\mathcal C}_{H_{L}}$ such that $F$ is left adjoint of $G$ and they induce a pair of inverse equivalences. In the Hopf quasigroup setting all Hopf module is strong, and then our results are the ones proved by Brzezi\'nski in \cite{Brz}. Also, in the weak Hopf case, all Hopf module is strong and then we generalize the theorem proved by B\"{o}hm, Nill and Szlach\'anyi in \cite{bohm}. 

\section{Weak Hopf quasigroups}
 
Throughout this paper $\mathcal C$ denotes a strict braided monoidal category with tensor product $\ot$, unit object $K$ and braiding $c$. For each object $M$ in  $ {\mathcal C}$, we denote the identity morphism by $id_{M}:M\rightarrow M$ and, for simplicity of notation, given objects $M$, $N$ and $P$ in ${\mathcal C}$ and a morphism $f:M\rightarrow N$, we write $P\ot f$ for $id_{P}\ot f$ and $f \ot P$ for $f\ot id_{P}$. We want to point out that there is no loss of generality in assuming that ${\mathcal C}$ is strict because by Theorem 3.5  of \cite{Christian} (which implies the Mac Lane's coherence theorem)  every monoidal category is monoidally equivalent to a strict one. This lets us to treat monoidal categories
as if they were strict and, as a consequence, the results proved in this paper hold for every non-strict symmetric monoidal category.

From now on we also assume in ${\mathcal C}$ that every idempotent morphism splits, i.e., if $\nabla:Y\rightarrow Y$ is such that $\nabla=\nabla\co\nabla$, there exist an object $Z$ and morphisms $i:Z\rightarrow Y$ and $p:Y\rightarrow Z$ such that $\nabla=i\co p$ and $p\co i =id_{Z}$. Note that, in these conditions, $Z$, $p$ and $i$ are unique up to isomorphism. There is
no loss of generality in assuming that ${\mathcal C}$ admits split idempotents, taking into account that, for a given category ${\mathcal C}$, there exists an universal embedding ${\mathcal
C}\rightarrow \hat{\mathcal C}$ such that $\hat{\mathcal C}$ admits
split idempotents, as was proved in  \cite{Karoubi}. The categories satisfying this property constitute a broad
class that includes, among others, the categories with epi-monic decomposition for
morphisms and categories with equalizers or coequalizers.

\begin{definition}
By a unital  magma in ${\mathcal C}$ we understand a triple $A=(A, \eta_{A}, \mu_{A})$ where $A$ is an object in ${\mathcal C}$ and $\eta_{A}:K\rightarrow A$ (unit), $\mu_{A}:A\ot A \rightarrow A$ (product) are morphisms in ${\mathcal C}$ such that $\mu_{A}\co (A\ot \eta_{A})=id_{A}=\mu_{A}\co (\eta_{A}\ot A)$. If $\mu_{A}$ is associative, that is, $\mu_{A}\co (A\ot \mu_{A})=\mu_{A}\co (\mu_{A}\ot A)$, the unital magma will be called a monoid in ${\mathcal C}$.   Given two unital magmas
(monoids) $A= (A, \eta_{A}, \mu_{A})$ and $B=(B, \eta_{B}, \mu_{B})$, $f:A\rightarrow B$ is a morphism of unital magmas (monoids)  if $\mu_{B}\co (f\ot f)=f\co \mu_{A}$ and $ f\co \eta_{A}= \eta_{B}$. 

By duality, a counital comagma in ${\mathcal C}$ is a triple ${D} = (D, \varepsilon_{D}, \delta_{D})$ where $D$ is an object in ${\mathcal C}$ and $\varepsilon_{D}: D\rightarrow K$ (counit), $\delta_{D}:D\rightarrow D\ot D$ (coproduct) are morphisms in ${\mathcal C}$ such that $(\varepsilon_{D}\ot D)\co \delta_{D}= id_{D}=(D\ot \varepsilon_{D})\co \delta_{D}$. If $\delta_{D}$ is coassociative, that is, $(\delta_{D}\ot D)\co \delta_{D}= (D\ot \delta_{D})\co \delta_{D}$, the counital comagma will be called a comonoid. If ${D} = (D, \varepsilon_{D}, \delta_{D})$ and ${ E} = (E, \varepsilon_{E}, \delta_{E})$ are counital comagmas
(comonoids), $f:D\rightarrow E$ is  a morphism of counital comagmas (comonoids) if $(f\ot f)\co \delta_{D} =\delta_{E}\co f$ and  $\varepsilon_{E}\co f =\varepsilon_{D}.$

If  $A$, $B$ are unital magmas (monoids) in ${\mathcal C}$, the object $A\ot B$ is a unital  magma (monoid) in ${\mathcal C}$ where $\eta_{A\ot B}=\eta_{A}\ot \eta_{B}$ and $\mu_{A\ot B}=(\mu_{A}\ot \mu_{B})\co (A\ot c_{B,A}\ot B).$  In a dual way, if $D$, $E$ are counital comagmas (comonoids) in ${\mathcal C}$, $D\ot E$ is a  counital comagma (comonoid) in ${\mathcal C}$ where $\varepsilon_{D\ot E}=\varepsilon_{D}\ot \varepsilon_{E}$ and $\delta_{D\ot
E}=(D\ot c_{D,E}\ot E)\co( \delta_{D}\ot \delta_{E}).$

Finally, if $D$ is a comagma and $A$ a magma, given two morphisms $f,g:D\rightarrow A$ we will denote by $f\ast g$ its convolution product in ${\mathcal C}$, that is 
$$f\ast g=\mu_{A}\co (f\ot g)\co \delta_{D}.$$

\end{definition}

Now we recall the notion of weak Hopf quasigroup in a braided monoidal category that we introduced in \cite{Asian}.

\begin{definition}
\label{Weak-Hopf-quasigroup}
{\rm A weak Hopf quasigroup $H$   in ${\mathcal C}$ is a unital magma $(H, \eta_H, \mu_H)$ and a comonoid $(H,\varepsilon_H, \delta_H)$ such that the following axioms hold:

\begin{itemize}

\item[(a1)] $\delta_{H}\co \mu_{H}=(\mu_{H}\ot \mu_{H})\co \delta_{H\ot H}.$

\item[(a2)] $\varepsilon_{H}\co \mu_{H}\co (\mu_{H}\ot H)=\varepsilon_{H}\co \mu_{H}\co (H\ot \mu_{H})$

\item[ ]$= ((\varepsilon_{H}\co \mu_{H})\ot (\varepsilon_{H}\co \mu_{H}))\co (H\ot \delta_{H}\ot H)$ 

\item[ ]$=((\varepsilon_{H}\co \mu_{H})\ot (\varepsilon_{H}\co \mu_{H}))\co (H\ot (c_{H,H}^{-1}\co\delta_{H})\ot H).$

\item[(a3)]$(\delta_{H}\ot H)\co \delta_{H}\co \eta_{H}=(H\ot \mu_{H}\ot H)\co ((\delta_{H}\co \eta_{H}) \ot (\delta_{H}\co \eta_{H}))$  \item[ ]$=(H\ot (\mu_{H}\co c_{H,H}^{-1})\ot H)\co ((\delta_{H}\co \eta_{H}) \ot (\delta_{H}\co \eta_{H})).$

\item[(a4)] There exists  $\lambda_{H}:H\rightarrow H$ in ${\mathcal C}$ (called the antipode of $H$) such that, if we denote the morphisms $id_{H}\ast \lambda_{H}$ by  $\Pi_{H}^{L}$ (target morphism) and $\lambda_{H}\ast id_{H}$ by $\Pi_{H}^{R}$ (source morphism),

\begin{itemize}

\item[(a4-1)] $\Pi_{H}^{L}=((\varepsilon_{H}\co \mu_{H})\ot H)\co (H\ot c_{H,H})\co ((\delta_{H}\co \eta_{H})\ot
H).$

\item[(a4-2)] $\Pi_{H}^{R}=(H\ot(\varepsilon_{H}\co \mu_{H}))\co (c_{H,H}\ot H)\co (H\ot (\delta_{H}\co \eta_{H})).$

\item[(a4-3)]$\lambda_{H}\ast \Pi_{H}^{L}=\Pi_{H}^{R}\ast \lambda_{H}= \lambda_{H}.$

\item[(a4-4)] $\mu_H\co (\lambda_H\ot \mu_H)\co (\delta_H\ot H)=\mu_{H}\co (\Pi_{H}^{R}\ot H).$

\item[(a4-5)] $\mu_H\co (H\ot \mu_H)\co (H\ot \lambda_H\ot H)\co (\delta_H\ot H)=\mu_{H}\co (\Pi_{H}^{L}\ot H).$

\item[(a4-6)] $\mu_H\co(\mu_H\ot \lambda_H)\co (H\ot \delta_H)=\mu_{H}\co (H\ot \Pi_{H}^{L}).$

\item[(a4-7)] $\mu_H\co (\mu_H\ot H)\co (H\ot \lambda_H\ot H)\co (H\ot \delta_H)=\mu_{H}\co (H\ot \Pi_{H}^{R}).$

\end{itemize}

\end{itemize}

Note that, if in the previous definition the triple $(H, \eta_H, \mu_H)$ is a monoid, we obtain the notion of weak Hopf algebra in a braided monoidal category. Then, if ${\mathcal C}$ is the category of vector spaces over a field ${\Bbb F}$, we have the monoidal version of the original definition of weak Hopf algebra introduced by B\"{o}hm, Nill and Szlach\'anyi in \cite{bohm}. On the other hand, under these conditions, if  $\varepsilon_H$ and $\delta_H$ are  morphisms of unital magmas (equivalently, $\eta_{H}$, $\mu_{H}$ are morphisms of counital comagmas), $\Pi_{H}^{L}=\Pi_{H}^{R}=\eta_{H}\ot \varepsilon_{H}$. As a consequence, conditions (a2), (a3), (a4-1)-(a4-3) trivialize, and we get the notion of Hopf quasigroup defined  by Klim and Majid in \cite{Majidesfera} in the category of vector spaces over a field ${\Bbb F}$.
}

\end{definition}

Below we will summarize the main properties of weak Hopf quasigroups. There are more, and the interested reader can see a complete list with the proofs  in \cite{Asian}.

First note that, by Propositions 3.1 and 3.2 of \cite{Asian}, the following equalities 
\begin{equation}
\label{pi-l}
\Pi_{H}^{L}\ast id_{H}=id_{H}\ast \Pi_{H}^{R}=id_{H},
\end{equation}
hold. Moreover, the antipode  is unique, $\lambda_{H}\co \eta_{H}=\eta_{H}$,  $\varepsilon_{H}\co\lambda_{H}=\varepsilon_{H}$, and, by Theorem 3.19 of \cite{Asian}, we have that it is antimultiplicative and anticomultiplicative. Also, if we define the morphisms $\overline{\Pi}_{H}^{L}$ and $\overline{\Pi}_{H}^{R}$ by 
$\overline{\Pi}_{H}^{L}=(H\ot (\varepsilon_{H}\co \mu_{H}))\co ((\delta_{H}\co \eta_{H})\ot H)$, $\;\overline{\Pi}_{H}^{R}=((\varepsilon_{H}\co \mu_{H})\ot H)\co (H\ot (\delta_{H}\co \eta_{H})),$
we proved in Proposition 3.4 of \cite{Asian}, that $\Pi_{H}^{L}$, $\Pi_{H}^{R}$, $\overline{\Pi}_{H}^{L}$ and 
$\overline{\Pi}_{H}^{R}$ are idempotent. On the other hand, Propositions 3.5, 3.7 and 3.9 of \cite{Asian} assert that 
\begin{equation}
\label{mu-pi-l}
\mu_{H}\co (H\ot \Pi_{H}^{L})=((\varepsilon_{H}\co
\mu_{H})\ot H)\co (H\ot c_{H,H})\co (\delta_{H}\ot
H),
\end{equation}
\begin{equation}
\label{mu-pi-r}
\mu_{H}\co (\Pi_{H}^{R}\ot H)=(H\ot(\varepsilon_{H}\co \mu_{H}))\co (c_{H,H}\ot H)\co
(H\ot \delta_{H}),
\end{equation}
\begin{equation}
\label{mu-pi-l-var}
\mu_{H}\co (H\ot \overline{\Pi}_{H}^{L})=(H\ot (\varepsilon_{H}\co
\mu_{H}))\co (\delta_{H}\ot H),
\end{equation}
\begin{equation}
\label{mu-pi-r-var}
\mu_{H}\co (\overline{\Pi}_{H}^{R}\ot H)=((\varepsilon_{H}\co
\mu_{H})\ot H)\co (H\ot \delta_{H}), 
\end{equation}
\begin{equation}
\label{delta-pi-l}
 (H\ot \Pi_{H}^{L})\co \delta_{H}=(\mu_{H}\ot H)\co (H\ot c_{H,H})\co ((\delta_{H}\co \eta_{H})\ot
H),
\end{equation}
\begin{equation}
\label{delta-pi-r}
(\Pi_{H}^{R}\ot H)\co \delta_{H}=(H\ot \mu_{H})\co (c_{H,H}\ot H)\co
(H\ot (\delta_{H}\co \eta_{H})),
\end{equation}
\begin{equation}
\label{delta-pi-l-var}
(\overline{\Pi}_{H}^{L}\ot H)\co \delta_{H}=(H\ot 
\mu_{H})\co ((\delta_{H}\co \eta_{H})\ot H),
\end{equation}
\begin{equation}
\label{delta-pi-r-var}
 (H\ot \overline{\Pi}_{H}^{R})\co \delta_{H}=(
\mu_{H}\ot H)\co (H\ot (\delta_{H}\co \eta_{H})), 
\end{equation}
hold. Also, it is possible to prove the following identities involving the idempotent morphisms $\Pi_{H}^{L}$, $\Pi_{H}^{R}$, $\overline{\Pi}_{H}^{L}$,  $\overline{\Pi}_{H}^{R}$ and the antipode $\lambda_{H}$  (see  Propositions 3.11 and 3.12 of \cite{Asian}): 
\begin{equation}
\label{pi-composition-1}
\Pi_{H}^{L}\co
\overline{\Pi}_{H}^{L}=\Pi_{H}^{L},\;\Pi_{H}^{L}\co
\overline{\Pi}_{H}^{R}=\overline{\Pi}_{H}^{R},\;\;\;\overline{\Pi}_{H}^{L}\co
\Pi_{H}^{L}=\overline{\Pi}_{H}^{L},\;\;\;\overline{\Pi}_{H}^{R}\co
\Pi_{H}^{L}=\Pi_{H}^{L},\;\;\;
\end{equation}
\begin{equation}
\label{pi-composition-3}
\Pi_{H}^{R}\co
\overline{\Pi}_{H}^{L}=\overline{\Pi}_{H}^{L},\;\;\;
\Pi_{H}^{R}\co
\overline{\Pi}_{H}^{R}=\Pi_{H}^{R},\;\;\;
\overline{\Pi}_{H}^{L}\co
\Pi_{H}^{R}=\Pi_{H}^{R},\;\;\; \overline{\Pi}_{H}^{R}\co
\Pi_{H}^{R}=\overline{\Pi}_{H}^{R}, 
\end{equation}
\begin{equation}
\label{pi-antipode-composition-1}
\Pi_{H}^{L}\co \lambda_{H}=\Pi_{H}^{L}\co
\Pi_{H}^{R}= \lambda_{H}\co \Pi_{H}^{R},\;\;\;
\Pi_{H}^{R}\co
\lambda_{H}=\Pi_{H}^{R}\co \Pi_{H}^{L}= \lambda_{H}\co
\Pi_{H}^{L},
\end{equation}
\begin{equation}
\label{pi-antipode-composition-3}
\Pi_{H}^{L}=\overline{\Pi}_{H}^{R}\co
\lambda_{H}=\lambda_{H}
\co\overline{\Pi}_{H}^{L},\;\;\;\Pi_{H}^{R}=
\overline{\Pi}_{H}^{L}\co \lambda_{H}=\lambda_{H} \co
\overline{\Pi}_{H}^{R}. 
\end{equation}

Moreover, by Proposition 3.16  of \cite{Asian},  we have
\begin{equation}
\label{mu-assoc-1}
\mu_{H}\co (\mu_{H}\ot H)\co (H\ot ((\Pi_{H}^{L}\ot H)\co \delta_{H}))=\mu_{H}=
\mu_{H}\co (\mu_{H}\ot \Pi_{H}^{R})\co (H\ot  \delta_{H}),
\end{equation}
\begin{equation}
\label{mu-assoc-2}
\mu_{H}\co (\Pi_{H}^{L}\ot \mu_{H})\co (\delta_{H}\ot H)=\mu_{H}=
\mu_{H}\co (H\ot (\mu_{H}\co ( \Pi_{H}^{R}\ot H)))\co (\delta_{H}\ot H).
\end{equation}

On the other hand, if $H_{L}=Im(\Pi_{H}^{L})$, 
$p_{L}:H\rightarrow H_{L}$, and $i_{L}:H_{L}\rightarrow H$ are the
morphisms such that $\Pi_{H}^{L}=i_{L}\co p_{L}$ and $p_{L}\co
i_{L}=id_{H_{L}}$, 
$$
\setlength{\unitlength}{3mm}
\begin{picture}(30,4)
\put(3,2){\vector(1,0){4}} \put(11,2.5){\vector(1,0){10}}
\put(11,1.5){\vector(1,0){10}} \put(1,2){\makebox(0,0){$H_{L}$}}
\put(9,2){\makebox(0,0){$H$}} \put(24,2){\makebox(0,0){$H\ot H$}}
\put(5.5,3){\makebox(0,0){$i_{L}$}} \put(16,3.5){\makebox(0,0){$
\delta_{H}$}} \put(16,0.5){\makebox(0,0){$(H\ot \Pi_{H}^{L}) \co
\delta_{H}$}}
\end{picture}
$$
is an equalizer diagram  and
$$
\setlength{\unitlength}{1mm}
\begin{picture}(101.00,10.00)
\put(20.00,8.00){\vector(1,0){25.00}}
\put(20.00,4.00){\vector(1,0){25.00}}
\put(55.00,6.00){\vector(1,0){21.00}}
\put(32.00,11.00){\makebox(0,0)[cc]{$\mu_{H}$ }}
\put(33.00,0.00){\makebox(0,0)[cc]{$\mu_{H}\co (H\ot \Pi_{H}^{L})
$ }} \put(65.00,9.00){\makebox(0,0)[cc]{$p_{L}$ }}
\put(13.00,6.00){\makebox(0,0)[cc]{$ H\otimes H$ }}
\put(50.00,6.00){\makebox(0,0)[cc]{$ H$ }}
\put(83.00,6.00){\makebox(0,0)[cc]{$H_{L} $ }}
\end{picture}
$$
is a coequalizer diagram. As a consequence, $(H_{L},
\eta_{H_{L}}=p_{L}\co \eta_{H}, \mu_{H_{L}}=p_{L}\co \mu_{H}\co
(i_{L}\ot i_{L}))$ is a unital magma in ${\mathcal C}$ and $(H_{L},
\varepsilon_{H_{L}}=\varepsilon_{H}\co i_{L}, \delta_{H}=(p_{L}\ot
p_{L})\co \delta_{H}\co i_{L})$ is a comonoid in ${\mathcal C}$ (see Proposition 3.13  of \cite{Asian}). 
Surprisingly,  the product $\mu_{H_{L}}$ is associative because, by Proposition 2.4 of \cite{JPAA},  we have that 
\begin{equation}
\label{aux-1-monoid-hl}
\delta_{H}\co \mu_{H}\co (i_{L}\ot H)=(\mu_{H}\ot H)\co (i_{L}\ot \delta_{H}), 
\end{equation}
\begin{equation}
\label{aux-2-monoid-hl}
\delta_{H}\co \mu_{H}\co (H\ot i_{L})=(\mu_{H}\ot H)\co (H\ot c_{H,H})\co  (\delta_{H}\ot i_{L}). 
\end{equation}
and, as a consequence, the following identities hold
\begin{equation}
\label{monoid-hl-1}
\mu_{H}\co ((\mu_{H}\co (i_{L}\ot H))\ot H)=\mu_{H}\co (i_{L}\ot \mu_{H}), 
\end{equation}
\begin{equation}
\label{monoid-hl-2}
\mu_{H}\co (H\ot (\mu_{H}\co (i_{L}\ot H)))=\mu_{H}\co ((\mu_{H}\co (H\ot i_{L}))\ot H), 
\end{equation}
\begin{equation}
\label{monoid-hl-3}
\mu_{H}\co (H\ot (\mu_{H}\co (H\ot i_{L})))=\mu_{H}\co (\mu_{H}\ot i_{L}). 
\end{equation}
By Proposition 3.9 of \cite{Asian}, (\ref{monoid-hl-2}) and the equality $\Pi_{H}^{L}\co \mu_{H}\co (\Pi_{H}^{L}\ot \Pi_{H}^{L})=\mu_{H}\co (\Pi_{H}^{L}\ot \Pi_{H}^{L})$,  it is easy to show that $\mu_{H_{L}}\co (H_{L}\ot \mu_{H_{L}})=\mu_{H_{L}}\co (\mu_{H_{L}}\ot H_{L})$ and therefore the unital magma $H_{L}$ is a monoid in the category ${\mathcal C}$. 
 
\begin{definition}
If $B$ is a monoid in ${\mathcal C}$, we will say that $B$ is separable if  there exists a morphism $q_{B}:K\rightarrow B\ot B$ satisfying  $(\mu_{B}\ot B)\co (B\ot q_{B})= (B\ot \mu_{B})\co
(q_{B}\ot B)$ and $ \mu_{B}\co q_{B}=\eta_{B}.$  The morphism
$q_{B}$ is called the Casimir morphism of $B$.  If the first equality of the previous line  holds and there exists a morphism $\varepsilon_{B}:B\rightarrow K$ such that $(B\ot \varepsilon_{B})\co q_{B}=\eta_{B}=(\varepsilon_{B}\ot B)\co q_{B}$, we will say that $B$ is Frobenius.
\end{definition}

\begin{proposition}
Let $H$ be a weak Hopf quasigroup. The monoid $H_{L}$ is Frobenius separable.
Therefore, if ${\mathcal C}$ is the category of vector spaces over
a field ${\Bbb F}$, $H_{L}$ is semisimple.
\end{proposition}

\begin{proof} Let $q_{H_{L}}:K\rightarrow H_{L}\ot H_{L}$ and
be the morphism defined
by $q_{H_{L}}=((p_{L}\co \lambda_{H})\ot p_{L})\co \delta_{H}\co
\eta_{H}$. Then, using the same proof of the similar result proved for weak braided Hopf algebras in \cite{IND} (see Proposition 2.19.), we obtain that $q_{H_{L}}$ is the Casimir morphism
of $H_{L}$ because $(\mu_{H_{L}}\ot H_{L})\co (H_{L}\ot q_{H_{L}})= \delta_{H_{L}}= (H_{L}\ot \mu_{H_{L}})\co (q_{H_{L}}\ot H_{L})$ and $\mu_{H_{L}}\co q_{H_{L}}=\eta_{H_{L}}.$  Also, $(H_{L}\ot \varepsilon_{H_{L}})\co q_{H_{L}}=\eta_{H_{L}}=(\varepsilon_{H_{L}}\ot H_{L})\co q_{H_{L}}$ and $H_{L}$ is Frobenius. Finally, if ${\mathcal C}$ is the category of vector spaces over a field ${\Bbb F}$, the semisimple character for $H_{L}$ follows from \cite{PIE}. 
\end{proof}

Finally, if $H_{R}=Im(\Pi_{H}^{R})$,
$p_{R}:H\rightarrow H_{R}$, and $i_{R}:H_{R}\rightarrow H$ are the
morphisms such that $\Pi_{H}^{R}=i_{R}\co p_{R}$ and $p_{R}\co
i_{R}=id_{H_{R}}$, the pair $(H_{R}, i_{R})$ is the equalizer of $\delta_{H}$ and $(\Pi_{H}^{R}\ot H) \co
\delta_{H}$. Also the pair $(H_{R}, p_{R})$ is the coequalizer of $\mu_{H}$ and $\mu_{H}\co (\Pi_{H}^{R}\ot H)
$.  As a consequence, $(H_{R},
\eta_{H_{R}}=p_{R}\co \eta_{H}, \mu_{H_{R}}=p_{R}\co \mu_{H}\co
(i_{R}\ot i_{R}))$ is a unital magma in ${\mathcal C}$ and $(H_{R},
\varepsilon_{H_{R}}=\varepsilon_{H}\co i_{R}, \delta_{H}=(p_{R}\ot
p_{R})\co \delta_{H}\co i_{R})$ is a comonoid in ${\mathcal C}.$ In a similar way to (\ref{monoid-hl-1})- (\ref{monoid-hl-3}), we can obtain 
\begin{equation}
\label{monoid-hr-1}
\mu_{H}\co ((\mu_{H}\co (i_{R}\ot H))\ot H)=\mu_{H}\co (i_{R}\ot \mu_{H}), 
\end{equation}
\begin{equation}
\label{monoid-hr-2}
\mu_{H}\co (H\ot (\mu_{H}\co (i_{R}\ot H)))=\mu_{H}\co ((\mu_{H}\co (H\ot i_{R}))\ot H), 
\end{equation}
\begin{equation}
\label{monoid-hr-3}
\mu_{H}\co (H\ot (\mu_{H}\co (H\ot i_{R})))=\mu_{H}\co (\mu_{H}\ot i_{R}). 
\end{equation}
Then, it is easy to show that $\mu_{H_{R}}\co (H_{R}\ot \mu_{H_{R}})=\mu_{H_{R}}\co (\mu_{H_{R}}\ot H_{R})$ and therefore the unital magma $H_{R}$ is a monoid in ${\mathcal C}$.  As a consequence, $H_{R}$ is Frobenius separable with Casimir morphism $q_{H_{R}}=(p_{R}\ot (p_{R}\co \lambda_{H}))\co
\delta_{H}\co \eta_{H}.$ Therefore, if ${\mathcal C}$ is the category of vector spaces over
a field ${\Bbb F}$, $H_{R}$ is semisimple.

\section{Hopf modules, strong Hopf modules and categorical equivalences}

The definition of right-right $H$-Hopf module for a weak Hopf quasigroup $H$ was introduced in \cite{Asian}. If $H$ is a Hopf quasigroup and ${\mathcal C}$ is the symmetric monoidal category ${\Bbb F}-Vect$, we get the notion defined by Brzezi\'nski in \cite{Brz} for Hopf quasigroups.

\begin{definition}
\label{Hopf-module}
Let $H$ be a  weak Hopf quasigroup and $M$  an object  in ${\mathcal
C}$. We say that $(M, \phi_{M}, \rho_{M})$ is a right-right $H$-Hopf module if the following axioms hold:
\begin{itemize}
\item[(b1)] The pair $(M, \rho_{M})$ is a right $H$-comodule, i.e. $\rho_{M}:M \rightarrow M\ot H$ is a morphism 
such that $(M\ot \varepsilon_{H})\co \rho_{M}=id_{M}$ and $(\rho_{M}\ot H)\co \rho_{M}=(M\ot \delta_{H})\co \rho_{M}$.
\item[(b2)] The morphism $\phi_{M}:M\ot H\rightarrow M$ satisfies:
\begin{itemize}
\item[(b2-1)] $\phi_{M}\co (M\ot \eta_{H})=id_{M}.$
\item[(b2-2)] $\rho_{M}\co \phi_{M}=(\phi_{M}\ot \mu_{H})\co (M\ot c_{H,H}	\ot H)\co (\rho_{M}\ot \delta_{H})$, i.e.
$\phi_{M}$ is a morphism of right $H$-comodules with the codiagonal coaction on $M\ot H$.
\end{itemize}
\item[(b3)] $\phi_{M}\circ (\phi_{M}\ot \lambda_H)\circ (M\ot \delta_{H})=\phi_{M}\co (M\ot \Pi_{H}^{L}).$
\item[(b4)] $\phi_{M}\circ (\phi_{M}\ot H)\circ (M\ot \lambda_H\ot H)\circ (M\ot \delta_H)=\phi_{M}\co (M\ot \Pi_{H}^{R}).$
\item[(b5)] $\phi_{M}\circ (\phi_{M}\ot H)\circ (M\ot \Pi_{H}^{L}\ot H)\circ (M\ot \delta_H)=\phi_{M}.$
\end{itemize}
\end{definition}
Obviously, if $H$ is a  weak Hopf quasigroup, the triple $(H, \phi_{H}=\mu_{H}, \rho_{H}=\delta_{H})$ is a  right-right $H$-Hopf module. Moreover,  if $(M, \phi_{M}, \rho_{M})$ is a right-right $H$-Hopf module, the axiom (b5)  is equivalent to 
$ \phi_{M}\circ (\phi_{M}\ot \Pi_{H}^{R})\circ (M\ot \delta_{H})=\phi_{M}$. Also, composing in (b2-2) with $M\ot \eta_{H}$ and $M\ot \varepsilon_{H}$ we have that $\phi_{M}\circ (M\ot \Pi_{H}^{R})\circ \rho_{M}=id_{M},$
and if $(M, \phi_{M}, \rho_{M})$, $(N, \phi_{N}, \rho_{N})$ are right-right $H$-Hopf modules, and there  exists a right $H$-comodule isomorphism $\alpha:M\rightarrow N$, the triple $(M,\phi_{M}^{\alpha}=\alpha^{-1}\co \phi_{N}\co (\alpha\ot H), \rho_{M})$ is a right-right $H$-Hopf module (see Proposition 4.7 of \cite{Asian}).

By Proposition 4.3 of \cite{Asian}, we have that  for all $(M, \phi_{M}, \rho_{M})$ right-right $H$-Hopf module, the morphism $q_{M}:=\phi_{M}\co (M\ot \lambda_{H})\co \rho_{M}:M\rightarrow M$  satisfies  $\rho_{M}\co q_{M}=(M\ot \Pi_{H}^{L})\co\rho_{M}\co q_{M}$ and, as a consequence, $q_{M}$ is  idempotent. Moreover, if $M^{co H}$ (object of coinvariants)  is the image of $q_{M}$ and $p_{M}:M\rightarrow M^{co H}$, $i_{M}:M^{co H}\rightarrow M$ are the morphisms such that $q_{M}=i_{M}\co p_{M}$ and 
 $id_{M^{co H}}=p_{M}\co i_{M}$, 
$$
\setlength{\unitlength}{3mm}
\begin{picture}(30,4)
\put(3,2){\vector(1,0){4}} \put(11,2.5){\vector(1,0){10}}
\put(11,1.5){\vector(1,0){10}} \put(1,2){\makebox(0,0){$M^{co H}$}}
\put(9,2){\makebox(0,0){$M$}} \put(24,2){\makebox(0,0){$M\ot H$}}
\put(5.5,3){\makebox(0,0){$i_{M}$}}
\put(16,3.5){\makebox(0,0){$ \rho_{M}$}}
\put(16,0.15){\makebox(0,0){$(M\ot \Pi_{H}^{L})\co\rho_{M}$}}
\end{picture}
$$
is an equalizer diagram. Also, 
$$
\setlength{\unitlength}{3mm}
\begin{picture}(30,4)
\put(3,2){\vector(1,0){4}} \put(11,2.5){\vector(1,0){10}}
\put(11,1.5){\vector(1,0){10}} \put(1,2){\makebox(0,0){$M^{co H}$}}
\put(9,2){\makebox(0,0){$M$}} \put(24,2){\makebox(0,0){$M\ot H$}}
\put(5.5,3){\makebox(0,0){$i_{M}$}}
\put(16,3.5){\makebox(0,0){$ \rho_{M}$}}
\put(16,0.15){\makebox(0,0){$(M\ot \overline{\Pi}_{H}^{R})\co\rho_{M}$}}
\end{picture}
$$
is  an equalizer diagram. Moreover, the following identities hold (see Remark 4.4 of \cite{Asian}): 
\begin{equation}
\label{new-c5-2-1}
\phi_{M}\circ (q_{M}\ot H)\co \rho_{M}=id_{M},
\end{equation}
\begin{equation}
\label{new-c5-2-2}
\rho_{M}\co \phi_{M}\co (i_{M}\ot H)=(\phi_{M}\ot H)\co (i_{M}\ot \delta_{H}),
\end{equation}
\begin{equation}
\label{new-c5-2-3}
p_{M}\co \phi_{M}\co (i_{M}\ot H)=p_{M}\co \phi_{M}\co (i_{M}\ot \Pi_{H}^{L}).
\end{equation}

On the other hand, the morphism 
$$\nabla_{M}:=(p_{M}\ot H)\co \rho_{M}\co \phi_{M}\co (i_{M}\ot H):M^{co H}\ot H\rightarrow M^{co H}\ot H$$
is idempotent and the equalities 
\begin{equation}
\label{tensor-idempotent-1}
\nabla_{M}=((p_{M}\co \phi_{M})\ot H)\co  (i_{M}\ot \delta_{H}), 
\end{equation}
\begin{equation}
\label{tensor-idempotent-2}
(M^{co H}\ot \delta_{H})\co \nabla_{M}= (\nabla_{M}\ot H)\co (M^{co H}\ot \delta_{H}).
\end{equation}
hold (see Proposition 4.5 of \cite{Asian}). If we define the morphisms
$$\omega_{M}:M^{co H}\ot H\rightarrow M,\;\;\;\;\;
\omega_{M}^{\prime}:M\rightarrow M^{co H}\ot H,$$  by
$\omega_{M}=\phi_{M}\co (i_{M}\otimes H)$ and
$\omega_{M}^{\prime}=(p_{M}\otimes H)\co \rho_{M}$. Then, 
 $\omega_{M}\circ \omega_{M}^{\prime}=id_{M}$ and
$\nabla_{M}=\omega_{M}^{\prime}\co \omega_{M}$. Also, we have a commutative diagram
$$
\unitlength=1mm \special{em:linewidth 0.4pt} \linethickness{0.4pt}
\begin{picture}(64.00,37.00)
\put(10.00,20.00){\vector(1,0){30.00}}
\put(8.00,18.00){\vector(4,-3){16.00}}
\put(27.00,6.00){\vector(1,1){12.00}}
\put(8.00,22.00){\vector(3,2){16.00}}
\put(26.00,32.00){\vector(4,-3){12.00}}

\put(0.00,20.00){\makebox(0,0)[cc]{$M^{co H}\ot H$}}
\put(50.00,20.00){\makebox(0,0)[cc]{$M^{co H}\ot H$}}
\put(24.00,36.00){\makebox(0,0)[cc]{$M$}}
\put(24.00,2.00){\makebox(0,0)[cc]{$M^{co H}\times H$}}
\put(13.00,29.00){\makebox(0,0)[cc]{$\omega_{M}$}}
\put(35.00,30.00){\makebox(0,0)[cc]{$\omega_{M}^{\prime}$}}
\put(9.00,10.00){\makebox(0,0)[cc]{$p_{M^{co H}\ot H}$}}
\put(39.00,9.00){\makebox(0,0)[cc]{$i_{M^{co H}\ot H}$}}
\put(23.00,23.00){\makebox(0,0)[cc]{$\nabla_{M}$}}

\end{picture}
$$
where $M^{co H}\times H$ denotes the image of $\nabla_{M}$ and $p_{M^{co H}\ot H}$,  $i_{M^{co H}\ot H}$ are the morphisms such that $p_{M^{co H}\ot H}\co i_{M^{co H}\ot H}=id_{M^{co H}\times H}$ and $i_{M^{co H}\ot H}\co p_{M^{co H}\ot H}=\nabla_{M}$. Therefore,
the morphism $$\alpha_{M}=p_{M^{co H}\ot H}\circ \omega^{\prime}_{M}:M\rightarrow M^{co H}\times H$$ is an
isomorphism of right $H$-modules (i.e. $\rho_{M^{co H}\times H}\co \alpha_{M}=(\alpha_{M}\ot H)\co \rho_{M}$)
with inverse $\alpha^{-1}_{M}=\omega_{M}\circ i_{M^{co H}\ot H}$. The
comodule structure of $M^{co H}\times H$ is the one induced by
the isomorphism $\alpha_{M}$ and it is equal to
$$
\rho_{M^{co H}\times H}=(p_{M^{co H}\ot H}\ot
H)\co (M^{co H}\ot \delta_{H})\co i_{M^{co H}\ot H}.
$$
As a consequence, the triple $(M^{co H}\times H, \phi_{M^{co H}\times H}, \rho_{M^{co H}\times H})$ where 
$$\phi_{M^{co H}\times H}=p_{M^{co H}\ot H}\co (M^{co H}\ot \mu_{H})\co (i_{M^{co H}\ot H}\ot H),$$
is a right-right $H$-Hopf module (see Proposition 4.8 of \cite{Asian}).

Finally, following Proposition 4.9 of \cite{Asian}, for the isomorphism of right $H$-comodules  $\alpha_{M}$, the triple $(M, \phi_{M}^{\alpha_{M}}, \rho_{M})$  is a right-right $H$-Hopf module  with the same object of coinvariants of  $(M, \phi_{M}, \rho_{M})$. Moreover, the identity  $\phi_{M}^{\alpha_{M}}=\phi_{M}\co (q_{M}\ot \mu_{H})\co (\rho_{M}\ot H)$ holds and 
\begin{equation}
\label{idemp-m-alfa}
q_{M}^{\alpha_{M}}=q_{M},
\end{equation} 
where $q_{M}^{\alpha_{M}}=\phi_{M}^{\alpha_{M}}\co (M\ot \lambda_{H})\co \rho_{M}$ is the idempotent morphism associated to the 
Hopf module $(M, \phi_{M}^{\alpha_{M}}, \rho_{M})$. Finally, if  $\nabla_{M}^{\alpha_{M}}$, denotes the idempotent morphism associated to $(M, \phi_{M}^{\alpha_{M}}, \rho_{M})$, we have that  $\nabla_{M}^{\alpha_{M}}=\nabla_{M}$
and then, for $(M, \phi_{M}^{\alpha_{M}}, \rho_{M})$, the associated isomorphism between $M$ and $M^{co H}\times H$ is $\alpha_{M}$. Finally,  $(\phi_{M}^{\alpha_{M}})^{\alpha_{M}}=
\phi_{M}^{\alpha_{M}}$ holds. Note that the triple $(H, \phi_{H}=\mu_{H}, \rho_{H}=\delta_{H})$ is a right-right $H$-Hopf module and $\phi_{H}^{\alpha_{H}}=\phi_{H}$.

\begin{proposition}
Let $H$ be a weak Hopf quasigroup and let $(M, \phi_{M}, \rho_{M})$ be a right-right  $H$-Hopf module. The following equality holds:
\begin{equation}
\label{(c2)}
\phi_{M}\co (i_{M}\ot \mu_{H})=\phi_{M}\co (i_{M}\ot \mu_{H}) \co (\nabla_{M}\ot H).
\end{equation}
\end{proposition}

\begin{proof} The equality holds because
\begin{itemize}
\item[ ]$\hspace{0.38cm}\phi_{M}\co (i_{M}\ot \mu_{H}) \co (\nabla_{M}\ot H)  $
\item[ ]$= \phi_{M}\co (q_{M}\ot \mu_{H}) \co ((\rho_{M}\co \phi_{M}\co (i_{M}\ot H))\ot H) $
\item[ ]$=  \phi_{M}\co (q_{M}\ot \mu_{H}) \co (((\phi_{M}\ot H)\co (i_{M}\ot \delta_{H}))\ot H) $
\item[ ]$=  \phi_{M}\co ((\phi_{M}\co (M\ot \lambda_{H}))\ot H) \co ((\rho_{M}\co \phi_{M}\co (i_{M}\ot H))\ot \mu_{H})\co (M^{coH}\ot \delta_{H}\ot H)$
\item[ ]$= \phi_{M}\co ((\phi_{M}\co (\phi_{M}\ot \lambda_{H})\co (M\ot \delta_{H}))\ot \mu_{H})\co (i_{M}\ot \delta_{H}\ot H)$
\item[ ]$= \phi_{M}\co ((\phi_{M}\co (M\ot \Pi_{H}^{L}))\ot \mu_{H})\co (i_{M}\ot \delta_{H}\ot H) $
\item[ ]$= \phi_{M}\co (\phi_{M}\ot \mu_{H})\co (i_{M}\ot (((\varepsilon_{H}\co \mu_{H})\ot H\ot H)\co \delta_{H\ot H}\co (\eta_{H}\ot H))\ot H) $
\item[ ]$= \phi_{M}\co (\phi_{M}\ot \mu_{H})\co (i_{M}\ot ((H\ot (\varepsilon_{H}\co \mu_{H})\ot H)\co ((c_{H,H}\co \delta_{H}\co \eta_{H})\ot \delta_{H}))\ot H) $
\item[ ]$= \phi_{M}\co (\phi_{M}\ot (\mu_{H}\co ((\mu_{H}\co (\overline{\Pi}_{H}^{R}\ot H))\ot H)))\co (i_{M}\ot 
(c_{H,H}\co \delta_{H}\co \eta_{H})\ot H\ot H)$
\item[ ]$= \phi_{M}\co (\phi_{M}\ot (\mu_{H}\co (\overline{\Pi}_{H}^{R}\ot \mu_{H})))\co (i_{M}\ot 
(c_{H,H}\co \delta_{H}\co \eta_{H})\ot H\ot H) $
\item[ ]$=\phi_{M}\co (\phi_{M}\ot ((  (\varepsilon_{H}\co \mu_{H})\ot H)\co (H\ot \delta_{H})))\co (i_{M}\ot (c_{H,H}\co \delta_{H}\co \eta_{H})\ot\mu_{H}) $
\item[ ]$=\phi_{M}\co (\phi_{M}\ot H)\co (M\ot ((\Pi_{H}^{L}\ot H)\co \delta_{H}))\co  (i_{M}\ot \mu_{H})  $
\item[ ]$=\phi_{M}\co (i_{M}\ot \mu_{H}),  $
\end{itemize}
where the first equality follows by the definition of $\nabla_{M}$, the second and the fourth ones follow by (\ref{new-c5-2-2}), the third one relies on the definition of $q_{M}$, and the fifth one is a consequence of (b3) of Definition \ref{Hopf-module}. In the sixth one we used the definition of $\Pi_{H}^{L}$ and the seventh and the eleventh  ones follow by the naturality of $c$. The equalities eighth and tenth are consequence of (\ref{mu-pi-r-var}), and the ninth one follows by $\mu_{H}\co ((\mu_{H}\co (\overline{\Pi}_{H}^{R}\ot H))\ot H)=\mu_{H}\co (\overline{\Pi}_{H}^{R}\ot \mu_{H}))$ (this equality holds by (\ref{monoid-hl-1}) and by (\ref{pi-composition-1}), more concretely, by $\Pi_{H}^{L}\co\overline{\Pi}_{H}^{R}=\overline{\Pi}_{H}^{R}$). Finally, the last one relies on (b5) of Definition \ref{Hopf-module}.
\end{proof}

\begin{definition}
\label{category}
Let $H$ be a weak Hopf quasigroup and let $(M, \phi_{M}, \rho_{M})$ and  $(N, \phi_{N}, \rho_{N})$ be right-right 
 $H$-Hopf modules. A morphism $f:M\rightarrow N$ in ${\mathcal C}$  is said to be $H$-quasilineal if the following identity holds:
\begin{equation}
\label{quasilineal}
\phi_{N}^{\alpha_{N}}\co (f\ot H)=f\co \phi_{M}^{\alpha_{M}}.
\end{equation}
A morphism of right-right $H$-Hopf modules between $M$ and $N$ is a morphism $f:M\rightarrow N$ in ${\mathcal C}$ such that is both a morphism of right $H$-comodules and $H$-quasilineal. The collection of all right $H$-Hopf modules with their morphisms forms a category  which will be  denoted by ${\mathcal M}^{H}_{H}$. 

\end{definition}

If $(M, \phi_{M}, \rho_{M})$ is an object in ${\mathcal M}^{H}_{H}$, for   $(M^{co H}\times H, \phi_{M^{co H}\times H}, \rho_{M^{co H}\times H})$ the  identity 
\begin{equation}
\label{quasilineal-1}
\phi_{M^{co H}\times H}^{\alpha_{M^{co H}\times H}}=\phi_{M^{co H}\times H}
\end{equation}
holds (see Proposition 4.12 of \cite{Asian}). Then, as a consequence, we can prove (see Theorem 4.13 of \cite{Asian})

\begin{theorem} (Fundamental Theorem of Hopf modules)
\label{fundamental}
Let $H$ be a weak Hopf quasigroup and  assume that $(M, \phi_{M}, \rho_{M})$ is an object in the category ${\mathcal M}^{H}_{H}$. Then, the  right-right $H$-Hopf modules $(M, \phi_{M}, \rho_{M})$ and $(M^{co H}\times H, \phi_{M^{co H}\times H}, \rho_{M^{co H}\times H})$ are isomorphic in ${\mathcal M}^{H}_{H}$.

\end{theorem}

From now on we assume that ${\mathcal C}$  admits coequalizers.  With ${\mathcal C}_{H_{L}}$ we will denote the category of right $H_{L}$-modules, i.e., the category whose objects are pairs $(N,\psi_{N})$ with $N$ an object in ${\mathcal C}$ and $
\psi_{N}:N\ot H_{L}\rightarrow N$ a morphism such that $\psi_{N}\co (N\ot \mu_{H_{L}})=\psi_{N}\co (\psi_{N}\ot H_{L})$, $\psi_{N}\co (N\ot\eta_{H_{L}})=id_{N}$. A morphism $f:(N,\psi_{N})\rightarrow (P,\psi_{P})$ in  ${\mathcal C}_{H_{L}}$ is a morphism $f:N\rightarrow P$ in ${\mathcal C}$ such that $\psi_{P}\co (f\ot H)=f\co \psi_{N}.$ Note that the pair $(H, \psi_{H}=\mu_{H}\co (H\ot i_{L}))$ is a right $H_{L}$-module.

Let $(N,\psi_{N})$ be an object in ${\mathcal C}_{H_{L}}$ and consider the coequalizer diagram 

\begin{equation}
\label{coequalizer-1}
\setlength{\unitlength}{1mm}
\begin{picture}(101.00,10.00)
\put(20.00,8.00){\vector(1,0){25.00}}
\put(20.00,4.00){\vector(1,0){25.00}}
\put(62.00,6.00){\vector(1,0){21.00}}
\put(32.00,11.00){\makebox(0,0)[cc]{$\psi_{N}\ot H$ }}
\put(33.00,0.00){\makebox(0,0)[cc]{$N\ot \varphi_{H}
$ }} 
\put(69.00,9.00){\makebox(0,0)[cc]{$n_{N}$ }}
\put(7.00,6.00){\makebox(0,0)[cc]{$ N\otimes H_{L}\ot H$ }}
\put(54.00,6.00){\makebox(0,0)[cc]{$ N\ot H$ }}
\put(96.00,6.00){\makebox(0,0)[cc]{$N\ot_{H_{L}} H $ }}
\end{picture}
\end{equation}
where $\varphi_{H}=\mu_{H}\co (i_{L}\ot H)$. By (\ref{aux-1-monoid-hl}) we have 
$$(n_{N}\ot H)\co (\psi_{N}\ot \delta_{H})=((n_{N}\co (N\ot \varphi_{H}))\ot H)\co (N\ot H_{L}\ot \delta_{H})=
(n_{N}\ot H)\co (N\ot (\delta_{H}\co \varphi_{H}))$$ 
and, as a consequence, there exists a unique morphism $\rho_{N\ot_{H_{L}} H}:N\ot_{H_{L}} H\rightarrow (N\ot_{H_{L}} H)\ot H$ such that 
\begin{equation}
\label{comodule}
\rho_{N\ot_{H_{L}} H}\co n_{N}=(n_{N}\ot H)\co (N\ot \delta_{H}).
\end{equation}
The pair $(N\ot_{H_{L}} H, \rho_{N\ot_{H_{L}} H})$ is a right $H$-comodule. Indeed: Trivially, $((N\ot_{H_{L}} H)\ot \varepsilon_{H})\co 
\rho_{N\ot_{H_{L}} H}=id_{N\ot_{H_{L}} H}$ because composing with $n_{N}$ we have 
$$((N\ot_{H_{L}} H)\ot \varepsilon_{H})\co 
\rho_{N\ot_{H_{L}} H}\co n_{N}=(n_{N}\ot \varepsilon_{H})\co (N\ot \delta_{H})=n_{N}.$$
Moreover, $(\rho_{N\ot_{H_{L}} H}\ot H)\co \rho_{N\ot_{H_{L}} H}=((N\ot_{H_{L}} H)\ot \delta_{H})\co \rho_{N\ot_{H_{L}} H}$ follows from 
$$(\rho_{N\ot_{H_{L}} H}\ot H)\co \rho_{N\ot_{H_{L}} H}\co n_{N}=(n_{N}\ot \delta_{H})\co (N\ot \delta_{H})=((N\ot_{H_{L}} H)\ot \delta_{H})\co \rho_{N\ot_{H_{L}} H}\co n_{N}.$$

On the other hand, by (\ref{monoid-hl-1}) we have 
$$n_{N}\co (\psi_{N}\ot \mu_{H})=n_{N}\co (N\ot (\mu_{H}\co (i_{L}\ot \mu_{H})))=n_{N}\co (N\ot (\mu_{H}\co (\varphi_{H}\ot H))),$$
and then, if the functor $-\ot H$ preserves coequalizers, there exists a unique morphism 
$$\phi_{N\ot_{H_{L}} H}:(N\ot_{H_{L}} H)\ot H\rightarrow N\ot_{H_{L}} H$$ such that 
\begin{equation}
\label{quasi-module}
\phi_{N\ot_{H_{L}} H}\co (n_{N}\ot H)=n_{N}\co (N\ot \mu_{H}).
\end{equation}

Trivially, $\phi_{N\ot_{H_{L}} H}\co ((N\ot_{H_{L}} H)\ot \eta_{H})=id_{N\ot_{H_{L}} H}$ because 
$$\phi_{N\ot_{H_{L}} H}\co (n_{N}\ot \eta_{H})=n_{N}\co (N\ot (\mu_{H}\co (H\ot \eta_{H})))=n_{N}.$$

Then, if the functor $-\ot H$ preserves coequalizers, the triple $(N\ot_{H_{L}} H, \phi_{N\ot_{H_{L}} H}, \rho_{N\ot_{H_{L}} H})$ is a right-right $H$-Hopf module. Indeed: By the previous reasoning  conditions 
(b1) and (b2-1) of Definition \ref{Hopf-module} hold. Composing with $n_{N}\ot H$ and using (\ref{comodule}), (\ref{quasi-module}) and (a1) of Definition \ref{Weak-Hopf-quasigroup} we have 
$$\rho_{N\ot_{H_{L}} H}\co \phi_{N\ot_{H_{L}} H}\co (n_{N}\ot H)=(n_{N}\ot H)\co (N\ot (\delta_{H}\co \mu_{H}))=(n_{N}\ot H)\co (N\ot ((\mu_{H}\ot \mu_{H})\co \delta_{H\ot H})$$
$$= (\phi_{N\ot_{H_{L}} H}\ot \mu_{H})\co ((N\ot_{H_{L}} H)\ot c_{H,H}\ot H)\co  ((\rho_{N\ot_{H_{L}} H}\co n_{N})\ot \delta_{H}),$$
and then (b2-2) of Definition \ref{Hopf-module} holds. Also, by (a4-6) of Definition \ref{Weak-Hopf-quasigroup}   and (\ref{quasi-module}) we obtain 
$$\phi_{N\ot_{H_{L}} H}\co (\phi_{N\ot_{H_{L}} H}\ot \lambda_{H})\co (n_{N}\ot \delta_{H})=
n_{N}\co (N\ot (\mu_{H}\co (\mu_{H}\ot \lambda_{H})\co (H\ot \delta_{H})))=n_{N}\co (N\ot (\mu_{H}\co (H\ot \Pi_{H}^{L})))$$
$$=\phi_{N\ot_{H_{L}} H}\co (n_{N}\ot \Pi_{H}^{L})$$
and then (b3) of Definition \ref{Hopf-module} holds. Similarly, by (\ref{quasi-module}) and (a4-7) of Definition \ref{Weak-Hopf-quasigroup}
we get (b4) of Definition \ref{Hopf-module}. The equality (b5) of this definition is a consequence of (\ref{quasi-module}) and (\ref{mu-assoc-1}).

Note that, by (\ref{comodule}), (\ref{quasi-module}), we obtain that 
\begin{equation}
\label{idem-strong}
q_{N\ot_{H_{L}} H}\co n_{N}=n_{N}\co (N\ot \Pi_{H}^{L}).
\end{equation}
Also, 
\begin{equation}
\label{action-induction}
\phi_{N\ot_{H_{L}} H}^{\alpha_{N\ot_{H_{L}} H}}=\phi_{N\ot_{H_{L}} H}
\end{equation}
because by 
(\ref{comodule}), (\ref{quasi-module})  and (\ref{mu-assoc-2}), 
$$\phi_{N\ot_{H_{L}} H}^{\alpha_{N\ot_{H_{L}} H}}\co (n_{N}\ot H)=n_{N}\co 
(N\ot (\mu_{H}\co (\Pi_{H}^{L}\ot \mu_{H})\co (\delta_{H}\ot H)))=n_{N}\co 
(N\ot \mu_{H})=\phi_{N\ot_{H_{L}} H}\co (n_{N}\ot H).$$

On the other hand, if $f:N\rightarrow P$ is a morphism in ${\mathcal C}_{H_{L}}$, we have that 
$$n_{P}\co (f\ot H)\co (\psi_{N}\ot H)=n_{P}\co (f\ot H)\co (N\ot \varphi_{H})$$
and, as a consequence, there exists an unique morphism $f\ot_{H_{L}}H:N\ot_{H_{L}} H\rightarrow P\ot_{H_{L}} H$ such that 
\begin{equation}
\label{mor-induction}
n_{P}\co (f\ot H)=(f\ot_{H_{L}}H)\co n_{N}.
\end{equation}
The morphism $f\ot_{H_{L}}H$
is a morphism in ${\mathcal M}^{H}_{H}$ because by (\ref{comodule}), (\ref{quasi-module}), (\ref{action-induction}) and (\ref{mor-induction}) 
$$\rho_{P\ot_{H_{L}} H}\co (f\ot_{H_{L}}H) \co n_{N}=(n_{P}\ot H)\co (f\ot \delta_{H})=((f\ot_{H_{L}}H)\ot H)\co 
\rho_{N\ot_{H_{L}} H}\co n_{N}$$
and 
$$  \phi_{P\ot_{H_{L}} H}^{\alpha_{P\ot_{H_{L}} H}}\co ((f\ot_{H_{L}}H)\ot H)\co (n_{N}\ot H)=\phi_{P\ot_{H_{L}} H}\co ((f\ot_{H_{L}}H)\ot H)\co (n_{N}\ot H)=n_{P}\co (f\ot \mu_{H})$$ $$=(f\ot_{H_{L}}H)\co 
\phi_{N\ot_{H_{L}} H}\co (n_{N}\ot H)=(f\ot_{H_{L}}H)\co 
\phi_{N\ot_{H_{L}} H}^{\alpha_{N\ot_{H_{L}} H}}\co (n_{N}\ot H).$$

Summarizing, we have the following proposition:

\begin{proposition}
\label{induction-functor}
Let $H$ be a weak Hopf quasigroup such that the functor $-\ot H$ preserves coequalizers. There exists a functor $$F:{\mathcal C}_{H_{L}}\rightarrow  {\mathcal M}^{H}_{H},$$ called the induction functor, defined on objects by $F((N,\psi_{N}))=(N\ot_{H_{L}} H, \phi_{N\ot_{H_{L}} H}, \rho_{N\ot_{H_{L}} H})$ and for morphisms by $F(f)=f\ot_{H_{L}}H$.
\end{proposition}

\begin{definition}
\label{strong-category}
Let $H$ be a weak Hopf quasigroup. With  ${\mathcal SM}^{H}_{H}$ we will denote the full subcategory of 
${\mathcal M}^{H}_{H}$ whose objects are the right-right $H$-Hopf modules $(M, \phi_{M}, \rho_{M})$ such that the following equality hold: 
\begin{equation}
\label{(c1)}
\phi_{M}\co ((\phi_{M}\co (M\ot i_{L}))\ot H)=\phi_{M}\co (M\ot (\mu_{H}\co (i_{L}\ot H))). 
\end{equation}
The objects of ${\mathcal SM}^{H}_{H}$ will be called right-right strong $H$-Hopf modules.

By  (\ref{monoid-hl-2}) we obtain that $(H, \phi_{H}=\mu_{H}, \rho_{H}=\delta_{H})$ is a  right-right strong $H$-Hopf module. Note that if $H$ is a Hopf quasigroup, (\ref{(c1)}) holds  because $i_{L}=\eta_{H}$ (see Theorem 1 of \cite{EmilioPura}). Then in this particular setting ${\mathcal SM}^{H}_{H}={\mathcal M}^{H}_{H}$. Also the previous equality holds trivially for any Hopf module associated to a weak (braided) Hopf algebra (see Section 3 of \cite{IND}). 
\end{definition}

\begin{proposition}
\label{fact-induction-functor}
Let $H$ be a weak Hopf quasigroup such that the functor $-\ot H$ preserves coequalizers. The induction functor $F:{\mathcal C}_{H_{L}}\rightarrow  {\mathcal M}^{H}_{H}$ factorizes through the category ${\mathcal SM}^{H}_{H}$.
\end{proposition}

\begin{proof}
We must show that for any $(N,\psi_{N})\in {\mathcal C}_{H_{L}}$, the triple $(N\ot_{H_{L}} H, \phi_{N\ot_{H_{L}} H}, \rho_{N\ot_{H_{L}} H})$ is an object in ${\mathcal SM}^{H}_{H}$. First note that if $-\ot H$ preserves coequalizers then $-\ot H_{L}$ preserves coequalizers, and (\ref{(c1)})  holds because by (\ref{quasi-module}) and (\ref{monoid-hl-2})
$$\phi_{N\ot_{H_{L}} H}\co ((\phi_{N\ot_{H_{L}} H}\co (n_{N}\ot i_{L}))\ot H)=n_{N}\co (N\ot (\mu_{H}\co 
((\mu_{H}\co (H\ot i_{L}))\ot H)))$$
$$=n_{N}\co (N\ot (\mu_{H}\co 
(H\ot (\mu_{H}\co  (i_{L}\ot H)))))=\phi_{N\ot_{H_{L}} H}\co (n_{N}\ot (\mu_{H}\co (i_{L}\ot H))).$$
\end{proof}

Let $(M,\phi_{M}, \rho_{M})$ be a right-right $H$-Hopf module. If $M$ is strong, the pair 
$$(M^{coH}, \psi_{M^{co H}}=p_{M}\co \phi_{M}\co (i_{M}\ot i_{L}))$$ is a right $H_{L}$-module. Indeed: Trivially $\psi_{M^{co H}}\co (M^{co H}\ot \eta_{H_{L}})=id_{M^{co H}}.$ Moreover, 
\begin{itemize}
\item[ ]$\hspace{0.38cm}  \psi_{M^{co H}} \co (\psi_{M^{co H}}\ot H_{L})$
\item[ ]$=p_{M}\co \phi_{M}\co ((\phi_{M}\co (\phi_{M}\ot \lambda_{H})\co (i_{M} \ot (\delta_{H}\co i_{L})))\ot i_{L})$
\item[ ]$=p_{M}\co \phi_{M}\co ((\phi_{M}\co (i_{M}\ot i_{L}))\ot i_{L}) $
\item[ ]$=p_{M}\co \phi_{M}\co ( i_{M}\ot (\mu_{H}\co (i_{L}\ot i_{L}))) $
\item[ ]$=\psi_{M^{co H}} \co (M^{co H}\ot \mu_{H_{L}}).$
\end{itemize}

The first equality follows by (\ref{new-c5-2-2}), the second one by (b3) of Definition \ref{Hopf-module}, the third one by (\ref{(c1)}), and the last one by the properties of $\mu_{H_{L}}$.

Let $g:M\rightarrow T$ be a morphism in  ${\mathcal SM}^{H}_{H}$. Using the comodule morphism condition we obtain that $\rho_{T}\circ g\circ i_{M}=(T\ot \overline{\Pi}_{H}^{R})\co \rho_{T}\circ g\circ i_{M}$ and this implies that there exists a unique morphism $g^{coH}:M^{coH}\rightarrow T^{coH}$ such that 
\begin{equation}
\label{coinv-morphism}
i_{T}\co g^{coH}=g\co i_{M}.
\end{equation}
Then, by (\ref{coinv-morphism}) and (\ref{idemp-m-alfa}),
$$i_{T}\co g^{coH}\co p_{M}=g\co q_{M}=g\co q_{M}^{\alpha_{M}}=q_{T}^{\alpha_{T}}\co g=q_{T}\co g$$
and, as a consequence,
\begin{equation}
\label{coinv-morphism-1}
g^{coH}\co p_{M}=p_{T}\co g.
\end{equation}

On the other hand, for any right-right $H$-Hopf module $M$,  by (\ref{idemp-m-alfa}), we know that $\nabla_{M}=\nabla_{M}^{\alpha_{M}}$. Then composing with $\phi_{M}\co (i_{M}\ot H)$ in this equality and using (\ref{new-c5-2-1}) we get the equality
\begin{equation}
\label{coinv-morphism-2}
\phi_{M}\co (i_{M}\ot H)=\phi_{M}^{\alpha_{M}}\co (i_{M}\ot H).
\end{equation}

Therefore, by (\ref{coinv-morphism-1}), (\ref{coinv-morphism-2}) and (\ref{coinv-morphism}) we obtain that 
$g^{coH}$ is a morphism of right $H_{L}$-modules because:
$$ g^{coH}\co \psi_{M^{coH}}=p_{T}\co g\co \phi_{M}\co (i_{M}\ot i_{L})=p_{T}\co g\co \phi_{M}^{\alpha_{M}}\co (i_{M}\ot i_{L})=p_{T}\co \phi_{T}^{\alpha_{T}}\co ((g\co i_{M})\ot i_{L})$$
$$=p_{T}\co \phi_{T}^{\alpha_{T}}\co ((i_{T}\co g^{coH})\ot i_{L})=p_{T}\co \phi_{T}\co ((i_{T}\co g^{coH})\ot i_{L})=\psi_{T^{coH}}\co (g^{coH}\ot H_{L}). $$

Thus, in this setting we have the following result.

\begin{proposition}
\label{coinvariant-functor}
Let $H$ be a weak Hopf quasigroup. There exists a functor $$G:{\mathcal SM}^{H}_{H}\rightarrow  {\mathcal C}_{H_{L}},$$ called the  functor of coinvariants, defined on objects by $G((M,\phi_{M},\rho_{M}))=(M^{coH}, \psi_{M^{coH}})$ and for morphisms by $G(g)=g^{coH}$.
\end{proposition}

\begin{proposition}
\label{aux-prop-iso}
Let $H$ be a weak Hopf quasigroup such that the functor $-\ot H$ preserves coequalizers. For any $(M,\phi_{M},\rho_{M})\in {\mathcal SM}^{H}_{H}$, the objects 
$M^{coH}\ot_{H_{L}}H$ and $M^{coH}\times H$ are isomorphic right-right $H$-Hopf modules.
\end{proposition}

\begin{proof} First note that $p_{M^{coH}\ot H}\co (\psi_{M^{coH}}\ot H)=p_{M^{coH}\ot H}\co (M^{coH}\ot \varphi_{H})$ because
\begin{itemize}
\item[ ]$\hspace{0.38cm}\nabla_{M}\co  (\psi_{M^{coH}}\ot H)  $
\item[ ]$=(p_{M}\ot H)\co \rho_{M}\co \phi_{M}\co ((q_{M}\co \phi_{M}\co (i_{M}\ot i_{L}))\ot H) $
\item[ ]$=(p_{M}\ot H)\co \rho_{M}\co \phi_{M}\co  ((\phi_{M}\co (\phi_{M}\ot \lambda_{H})\co (i_{M}\ot (\delta_{H}\co i_{L})))\ot H) $
\item[ ]$=(p_{M}\ot H)\co \rho_{M}\co \phi_{M}\co  ((\phi_{M}\co (i_{M}\ot  i_{L}))\ot H) $
\item[ ]$=(p_{M}\ot H)\co \rho_{M}\co \phi_{M}\co  (i_{M}\ot 	(\mu_{H}\co (i_{L}\ot H))) $
\item[ ]$= \nabla_{M}\co (M^{coH}\ot \varphi_{H}) .$
\end{itemize}

The first and the last equalities are consequence of the definition of $\nabla_{M}$. The second one follows by 
(\ref{new-c5-2-2}), the third one by (b3) of Definition \ref{Hopf-module} and the properties of $\Pi_{H}^{L}$. Finally, the fourth one relies on (\ref{(c1)}).

Let $t:M^{coH}\ot H\rightarrow P$ be a morphism such  that $t\co (\psi_{M^{coH}}\ot H)=t\co (M^{coH}\ot \varphi_{H})$. Put $t^{\prime}:M^{coH}\times H\rightarrow P$ defined by $t^{\prime}=t\co i_{M^{co H}\ot H}$. Then 
$$t^{\prime}\co p_{M^{co H}\ot H}=t\co \nabla_{M}=t$$
because
\begin{itemize}
\item[ ]$\hspace{0.38cm}t\co \nabla_{M}  $
\item[ ]$= t\co ((p_{M}\co \phi_{M})\ot H)\co (i_{M}\ot \delta_{H})$
\item[ ]$= t\co ((p_{M}\co \phi_{M})\ot H)\co (i_{M}\ot ((\Pi_{H}^{L}\ot H)\co\delta_{H}))$
\item[ ]$= t\co (\psi_{M^{coH}}\ot H)\co (M^{coH}\ot ((p_{L}\ot H)\co\delta_{H})) $
\item[ ]$= t\co  (M^{coH}\ot (\Pi_{H}^{L}\ast id_{H}))$
\item[ ]$= t.$
\end{itemize}
Applying (\ref{tensor-idempotent-1}) we obtain the first equality. The second one relies on (\ref{new-c5-2-3}). The third one follows by the definition of $\psi_{M^{coH}}$ and the fourth one by the properties of $t$. Finally the last one is a consequence of (\ref{pi-l}). 

The morphism $t^{\prime}$ is the unique such that $t^{\prime}\co p_{M^{co H}\ot H}=t$ because if $r:M^{coH}\times H\rightarrow P$ satisfies $r\co p_{M^{co H}\ot H}=t$, composing with $i_{M^{co H}\ot H}$, we obtain 
$r=t\co i_{M^{co H}\ot H}=t^{	\prime}.$ Therefore, 

$$
\setlength{\unitlength}{1mm}
\begin{picture}(101.00,10.00)
\put(20.00,8.00){\vector(1,0){25.00}}
\put(20.00,4.00){\vector(1,0){25.00}}
\put(65.00,6.00){\vector(1,0){21.00}}
\put(32.00,11.00){\makebox(0,0)[cc]{$\psi_{M^{co H}}\ot H$ }}
\put(33.00,0.00){\makebox(0,0)[cc]{$M^{co H}\ot \varphi_{H}
$ }} 
\put(76.00,9.00){\makebox(0,0)[cc]{$p_{M^{co H}\ot H}$ }}
\put(7.00,6.00){\makebox(0,0)[cc]{$ M^{co H}\otimes H_{L}\ot H$ }}
\put(56.00,6.00){\makebox(0,0)[cc]{$ M^{co H}\ot H$ }}
\put(98.00,6.00){\makebox(0,0)[cc]{$M^{co H}\times H $ }}
\end{picture}
$$

is a coequalizer diagram and as a consequence there exists an isomorphism $$s_{M}:M^{coH}\ot_{H_{L}}H\rightarrow M^{co H}\times H$$ 
such that 
\begin{equation}
\label{iso-aux}
s_{M}\co n_{M^{coH}}=p_{M^{co H}\ot H}.
\end{equation}

The morphism $s_{M}$ belongs to the category of right-right $H$-Hopf modules. Indeed: It is a morphism of right $H$-comodules because composing with $n_{M^{coH}}$  and using the equalities (\ref{iso-aux}), (\ref{tensor-idempotent-2}) we have 
$$\rho_{M^{co H}\times H}\co s_{M}\co n_{M^{coH}}=\rho_{M^{co H}\times H}\co p_{M^{co H}\ot H}
=(p_{M^{co H}\ot H}\ot H)\co (M^{coH}\ot \delta_{H})\co \nabla_{M}$$
$$=(p_{M^{co H}\ot H}\ot H)\co (M^{coH}\ot \delta_{H})=((s_{M}\co n_{M^{coH}})\ot H)\co (M^{coH}\ot \delta_{H})=(s_{M}\ot H)\co \rho_{M^{co H}\ot_{H_{L}} H}\co n_{M^{coH}}.$$
Moreover, by (\ref{quasilineal-1}) and (\ref{action-induction}) we know that $\phi_{M^{co H}\times H}^{\alpha_{M^{co H}\times H}}=\phi_{M^{co H}\times H}$ and $\phi_{M^{co H}\ot_{H_{L}} H}^{\alpha_{M^{co H}\ot_{H_{L}} H}}=\phi_{M^{co H}\ot_{H_{L}} H}$. As a consequence, $s_{M}$ is $H$-quasilineal because composing with the coequalizer  $n_{M^{coH}}\ot H$ and the equalizer $i_{M^{coH}\ot H}$ we obtain
\begin{itemize}
\item[ ]$\hspace{0.38cm} i_{M^{coH}\ot H}\co s_{M}\co  \phi_{M^{co H}\ot_{H_{L}} H}\co (n_{M^{coH}}\ot H)$
\item[ ]$=\omega_{M}^{\prime}\co \phi_{M}\co (i_{M}\ot \mu_{H})  $
\item[ ]$=\omega_{M}^{\prime}\co \phi_{M}\co (i_{M}\ot  \mu_{H})\co (\nabla_{M}\ot H)$
\item[ ]$= i_{M^{coH}\ot H}\co \phi_{M^{co H}\times H}\co ((s_{M}\co n_{M^{coH}})\ot H), $
\end{itemize}
where the first equality follows by (\ref{quasi-module}), the second one by (\ref{(c2)}), and the last one by (\ref{iso-aux}).

\end{proof}

\begin{theorem}
\label{principalofpaper}
For any Hopf quasigroup H such the the functor $-\ot H$ preserves coequalizers, the category  ${\mathcal SM}^{H}_{H}$  is  equivalent to the category ${\mathcal C}_{H_{L}}$.
\end{theorem}

\begin{proof} To prove the theorem we will obtain that the induction functor F is left adjoint to the coinvariants functor $G$ and that the unit and counit  associated to this adjunction are natural isomorphisms. Then we divide the proof in three steps. 

{\it \underline{Step 1:}} In this step we will define the unit of the adjunction. For any right $H_{L}$-module  $(N, \psi_{N})$ define $u_{N}:N\rightarrow GF(N)=(N\ot_{H_{L}}H)^{coH}$ as the unique morphism such that 
\begin{equation}
\label{unit}
i_{N\ot_{H_{L}}H}\co u_{N}=n_{N}\co (N\ot \eta_{H}). 
\end{equation}
This morphism exists and is unique because by (\ref{comodule}) and (\ref{delta-pi-r-var}) we have 
$$((N\ot_{H_{L}}H)\ot \overline{\Pi}_{H}^{R})\co \rho_{N\ot_{H_{L}}H}\co n_{N}\co (N\ot \eta_{H})=
(n_{N}\ot \overline{\Pi}_{H}^{R})\co (N\ot (\delta_{H}\co \eta_{H}))=(n_{N}\ot H)\co (N\ot (\delta_{H}\co \eta_{H}))
$$
$$=\rho_{N\ot_{H_{L}}H}\co n_{N}\co (N\ot \eta_{H}).$$

Also, it is a morphism in  ${\mathcal C}_{H_{L}}$. Indeed: Composing with the equalizer $i_{N\ot_{H_{L}}H}$ we have 

\begin{itemize}
\item[ ]$\hspace{0.38cm} i_{N\ot_{H_{L}}H} \co \psi_{N\ot_{H_{L}}H}\co (u_{N}\ot H_{L}) $
\item[ ]$=q_{N\ot_{H_{L}}H} \co \phi_{N\ot_{H_{L}}H}\co ((n_{N}\co (N\ot \eta_{H}))\ot i_{L})$
\item[ ]$=q_{N\ot_{H_{L}}H} \co n_{N}\co (N\ot (\mu_{H}\co (\eta_{H}\ot i_{L})))  $
\item[ ]$=q_{N\ot_{H_{L}}H} \co n_{N}\co (N\ot i_{L}) $
\item[ ]$= n_{N}\co (N\ot (\Pi_{H}^{L}\co i_{L})) $
\item[ ]$=n_{N}\co (N\ot  i_{L}) $
\item[ ]$=n_{N}\co (N\ot (\mu_{H}\co (i_{L}\ot \eta_{H}))) $
\item[ ]$=n_{N}\co (\psi_{N}\ot \eta_{H}) $
\item[ ]$=i_{N\ot_{H_{L}}H} \co u_{N}\co \psi_{N}, $
\end{itemize}

where the first and last equalities follow by (\ref{unit}), the second one by (\ref{quasi-module}), the third and the sixth one by the unit properties, the fourth one by (\ref{idem-strong}), the fifth one by the properties of $\Pi_{H}^{L}$ and, the seventh one is a consequence of the definition of $N\ot_{H_{L}}H$. 

The morphism $u_{N}$ is natural in $N$ because if $f:N\rightarrow P$ is a morphism in ${\mathcal C}_{H_{L}}$ by 
(\ref{coinv-morphism}), (\ref{unit}), (\ref{mor-induction}) we have 
$$ i_{P\ot_{H_{L}}H}\co (f\ot_{H_{L}}H)^{coH}\co u_{N}=(f\ot_{H_{L}}H)\co i_{N\ot_{H_{L}}H}\co u_{N}=
(f\ot_{H_{L}}H)\co n_{N}\co (N\ot \eta_{H})$$
$$=n_{P}\co (f\ot \eta_{H})=i_{P\ot_{H_{L}}H}\co u_{P}\co f,$$
and then $(f\ot_{H_{L}}H)^{coH}\co u_{N}=u_{P}\co f$.

Finally, we will prove that $u_{N}$ is an isomorphism for all right $H_{L}$-module $N$. First note that 
$\psi_{N}\co (\psi_{N}\ot p_{L})=\psi_{N}\co (N\ot (p_{L}\co\varphi_{H}))$ and then there exists a unique morphism $m_{N}:N\ot_{H_{L}}H\rightarrow N$ such that 
\begin{equation}
\label{mn}
m_{N}\co n_{N}=\psi_{N}\co (N\ot p_{L}). 
\end{equation}
Define $x_{N}=m_{N}\co i_{N\ot_{H_{L}}H}:(N\ot_{H_{L}}H)^{coH}\rightarrow N$. Then, composing with 
$i_{N\ot_{H_{L}}H}$ and $p_{N\ot_{H_{L}}H}\co n_{N}$ and using (\ref{idem-strong}), (\ref{mn}), (\ref{unit}) and the properties of $\Pi_{H}^{L}$ we have 
\begin{itemize}
\item[ ]$\hspace{0.38cm} i_{N\ot_{H_{L}}H} \co  u_{N}\co x_{N}\co p_{N\ot_{H_{L}}H}\co n_{N}$
\item[ ]$=i_{N\ot_{H_{L}}H} \co  u_{N}\co m_{N}\co q_{N\ot_{H_{L}}H}\co n_{N} $
\item[ ]$=i_{N\ot_{H_{L}}H} \co  u_{N}\co \psi_{N}\co (N\ot (p_{L}\co \Pi_{H}^{L}))  $
\item[ ]$=n_{N}\co ( (\psi_{N}\co (N\ot p_{L}))\ot \eta_{H})$
\item[ ]$= n_{N}\co (N\ot (\mu_{H}\co ( \Pi_{H}^{L}\ot \eta_{H})))$
\item[ ]$=q_{N\ot_{H_{L}}H} \co n_{N}. $
\end{itemize}
Therefore, $u_{N}\co x_{N}=id_{(N\ot_{H_{L}}H)^{coH}}$. Moreover, by (\ref{unit}) and  (\ref{mn}), 
$$x_{N}\co u_{N}=\psi_{N}\co (N\ot (p_{L}\co \eta_{H}))=id_{N}.$$

{\it \underline{Step 2:}} For any $(M,\phi_{M},\rho_{M})\in {\mathcal SM}^{H}_{H}$ the counit is defined by 
$$ v_{M}=\alpha_{M}^{-1}\co s_{M}:M^{coH}\ot_{H_{L}}H\rightarrow M,$$ 
where $\alpha_{M}^{-1}=\omega_{M}\circ i_{M^{co H}\ot H}$ is the  inverse of the isomorphism $\alpha_{M}$ defined in Theorem \ref{fundamental} and $s_{M}$ the isomorphism defined in Proposition \ref{aux-prop-iso}. Note that  $\alpha_{M}^{-1}$ and $s_{M}$ are isomorphisms in ${\mathcal SM}^{H}_{H}$ and then $v_{M}$ is an isomorphism in ${\mathcal SM}^{H}_{H}$. Also, $v_{M}$ is the unique morphism such that 
\begin{equation}
\label{counit}
v_{M}\co n_{M^{coH}}=\phi_{M}\co (i_{M}\ot H),
\end{equation}
because by (\ref{new-c5-2-2}), (b3) of Definition \ref{Hopf-module}, the properties of $\Pi_{H}^{L}$ and 
(\ref{(c1)}),
\begin{itemize}
\item[ ]$\hspace{0.38cm}\phi_{M}\co ((i_{M}\co \psi_{M^{coH}})\ot H)  $
\item[ ]$=\phi_{M}\co ((\phi_{M}\co (\phi_{M}\ot \lambda_{H})\co (i_{M}\ot (\delta_{H}\co i_{L})))\ot H)  $
\item[ ]$= \phi_{M}\co ((\phi_{M}\co  (i_{M}\ot (\Pi_{H}^{L}\co i_{L})))\ot H) $
\item[ ]$= \phi_{M}\co ((\phi_{M}\co  (i_{M}\ot  i_{L}))\ot H) $
\item[ ]$=\phi_{M}\co ( i_{M}\ot \varphi_{H}), $
\end{itemize}
and, on the other hand, by (\ref{iso-aux})
 $$v_{M}\co n_{M^{coH}}=\alpha_{M}^{-1}\co s_{M}\co n_{M^{coH}}=\omega_{M}\co
 \omega_{M}^{\prime}\co 
 \omega_{M}=\omega_{M}=\phi_{M}\co (i_{M}\ot H).$$

{\it \underline{Step 3:}} Now we prove  the triangular identities for the unit and the counit that we defined previously. Indeed:  The first triangular identity holds because composing with $n_{N}$ we have 
\begin{itemize}
\item[ ]$\hspace{0.38cm} v_{N\ot_{H_{L}}H}\co (u_{N}\ot_{H_{L}}H)\co n_{N} $
\item[ ]$=v_{N\ot_{H_{L}}H}\co n_{ (N\ot_{H_{L}}H)^{coH}}\co (u_{N}\ot H) $
\item[ ]$= \phi_{N\ot_{H_{L}}H}\co ((i_{N\ot_{H_{L}}H}\co u_{N})\ot H)$
\item[ ]$= \phi_{N\ot_{H_{L}}H}\co ((n_{N}\co (N\ot \eta_{H}))\ot H)$
\item[ ]$= n_{N}\co (N\ot (\mu_{H}\co (\eta_{H}\ot H))) $
\item[ ]$=n_{N}, $
\end{itemize}
where the first equality follows by (\ref{mor-induction}), the second one by (\ref{counit}), the third one by (\ref{unit}) and the fourth one by (\ref{quasi-module}). The last one follows by the properties of the unit $\eta_{H}$. 
Finally, if we compose with $i_{M}$, applying (\ref{coinv-morphism}), (\ref{unit}) and (\ref{counit}) we obtain:
\begin{itemize}
\item[ ]$\hspace{0.38cm} i_{M}\co v_{M}^{coH}\co u_{M^{coH}} $
\item[ ]$=v_{M}\co  i_{N\ot_{H_{L}}H}\co u_{M^{coH}}  $
\item[ ]$= v_{M}\co n_{M^{coH}}\co ( M^{coH}\ot \eta_{H})$
\item[ ]$=\phi_{M}\co (i_{M}\ot \eta_{H})  $
\item[ ]$=i_{M}, $
\end{itemize}
and then $v_{M}^{coH}\co u_{M^{coH}} =id_{M^{coH}}.$

\end{proof}

\section*{Acknowledgements}
The  authors were supported by  Ministerio de Econom\'{\i}a y Competitividad and by Feder founds. Project MTM2013-43687-P: Homolog\'{\i}a, homotop\'{\i}a e invariantes categ\'oricos en grupos y \'algebras no asociativas.


\begin{thebibliography}{99}


\bibitem{IND}
 J. N. Alonso \'Alvarez,  J.M. Fern\'{a}ndez Vilaboa,
R. Gonz\'{a}lez Rodr\'{\i}guez,  \textit{Weak braided Hopf algebras}, 
Indiana Univ. Math. J.   \textbf{57} (2008), 2423-2458.

\bibitem{Asian}
 J. N. Alonso \'Alvarez,  J.M. Fern\'{a}ndez Vilaboa,
R. Gonz\'{a}lez Rodr\'{\i}guez,  \textit{Weak Hopf quasigroups}, 
Asian Journal of Mathematics (2015) (in press) (available in arXiv:1410.2180).

\bibitem{JPAA}
 J. N. Alonso \'Alvarez,  J.M. Fern\'{a}ndez Vilaboa,
R. Gonz\'{a}lez Rodr\'{\i}guez,  \textit{Cleft and Galois extensions associated to a weak Hopf quasigroup}, 
arXiv:1412.1622 (2014).

\bibitem{BEN} J. B\'enabou, \textit{Introduction to bicategories}, in Reports of the Midwest Categorical Seminar,  LNM \textbf{47}, Springer, 1967, pp. 1-77.

\bibitem{bohm} G. B\"{o}hm, F. Nill, K. Szlach\'anyi,
 \textit{Weak Hopf algebras, I. Integral theory and
$C^{\ast}$-structure}, J.  Algebra, \textbf{221}  (1999), 385-438.(c

\bibitem{Brz}
T. Brzezi\'nski, \textit{Hopf modules and the fundamental theorem for Hopf (co)quasigroups}, Internat. Elec. J. Algebra,  \textbf{8} (2010), 114-128.

\bibitem{Hanna} H. Henker, \textit{Module categories over quasi-Hopf algebras and weak Hopf algebras and the projectivity of Hopf modules}, Thesis Dissertation, LMU Munich (2011) (available in http://edoc.ub.uni-muenchen.de/13148/)


\bibitem{Karoubi}
 M. Karoubi, \textit{K-th\'{e}orie},  Les Presses de
l'Universit\'{e} de Montr\'{e}al, Montr\'{e}al, 1971.

\bibitem{Christian} C. Kassel, \textit{Quantum Groups},  GTM \textbf{155}, Springer-Verlag, New York, 1995.

\bibitem{Majidesfera}
J. Klim, S. Majid, \textit{Hopf quasigroups and the algebraic 7-sphere},
J. Algebra \textbf{323} (2010), 3067-3110.


\bibitem{Larson-Sweedler} R.G. Larson, M.E. Sweedler, \textit{An associative orthogonal bilinear form for Hopf
algebras}, Amer. J. Math. \textbf{91} (1969),
75-93.


\bibitem{EmilioPura}
M.P. L\'opez, E. Villanueva,  \textit{The antipode and the (co)invariants of
a finite Hopf (co)quasigroup},  Appl. Cat. Struct. \textbf{21}
(2013), 237-247.

\bibitem{Mac} S. Mac Lane, \textit{Categories for the working mathematician}, GTM \textbf{5}, Springer-Verlag, New-York, 1971.
 

\bibitem{PI2}
J.M. P\'erez-Izquierdo, \textit{Algebras, hyperalgebras, nonassociative
bialgebras and loops},
 Adv. Math. \textbf{208} (2007), 834-876.
 
\bibitem{PIE} R. S. Pierce, \textit{Associative algebras}, GTM \textbf{88},
Springer-Verlag, New York, (1982).

\bibitem{Sweedler} M.E.  Sweedler, \textit{Hopf
algebras}, Benjamin, New York (1969).







\end{thebibliography}
\end{document}